\documentclass{svproc}

\usepackage{type1cm}        
\usepackage{graphicx}     

\usepackage[bottom]{footmisc}

\usepackage{nicefrac}
\usepackage{mathrsfs}
\usepackage{amsmath,amssymb,amscd} 
\usepackage{mathtools}
\usepackage{bbm}
\usepackage{caption} 
\usepackage{doi}
\usepackage[linesnumbered,ruled,vlined]{algorithm2e}
\usepackage{tikz}
\usetikzlibrary{decorations.markings,arrows.meta}

\newcommand{\T}{\mathcal{T}}

\newcommand{\X}{\mathcal{X}}

\newcommand{\N}{\mathbb{N}}
\newcommand{\R}{\mathbb{R}}

\newcommand{\one}{\mathbbm{1}}
\newcommand{\PP}{\mathbb{P}}

\newcommand{\ZZ}{\mathcal{Z}}
\newcommand{\Pcal}{\mathcal{P}}
\newcommand{\EE}{\mathbb{E}}

\newcommand{\Lscr}{\mathscr{L}}

\DeclareMathOperator*{\argmin}{argmin}

\usepackage{newtxmath}

\usepackage{url}

\begin{document}
	\mainmatter          
	
	\title{Lower and Upper Expected Hitting Times for Weighted Imprecise Markov Chains}
	
	\titlerunning{Expected Hitting Times for WIMCs}  
	
	\author{Marco Sangalli \and Thomas Krak}
	
	\authorrunning{Marco Sangalli et al.}
	
	\tocauthor{Marco Sangalli, Thomas Krak}
	
	\institute{Eindhoven University of Technology, Eindhoven,\\
		\email{m.sangalli@tue.nl}}
	
	\maketitle      
	
	\begin{abstract}
		In this paper, we extend hitting times for imprecise Markov chains to the framework of weighted imprecise Markov chains (WIMCs), in which each transition is associated with a strictly positive weight encoded by a matrix $W$.
        Given a convex set $\mathcal{T}$ of admissible transition matrices, we define lower and upper expected hitting times for WIMCs as the infimum and supremum of the (weighted) expected hitting times over $\mathcal{T}$, and we characterise these quantities as the unique solutions of nonlinear fixed-point equations. We show that any weighted hitting time problem can be transformed into an unweighted hitting time problem on an augmented state space, enabling the reuse of existing IMC theory and algorithms. In particular, we are able to adapt known iterative methods for the numerical computation of lower and upper hitting times for WIMCs. 
	\end{abstract}
    \section{Introduction and Motivation}
Hitting times are a fundamental quantitative measure for stochastic processes. Given a target set of states $A$, the expected hitting time measures how long a process needs, on average, to reach the target $A$. In classical precise Markov chains this notion is usually understood as the expected number of transitions until first arrival.  

A first generalisation replaces step counts with a matrix of positive {weights} $W$ that assigns a travel time (or cost) to each transition. The resulting object is a weighted Markov chain, and the corresponding hitting time is defined as the accumulated weight along a random trajectory up to the first entrance in $A$.
This naturally arises in many applications, ranging from routing and transportation to robotic surveillance~\cite{Carron2016Weighted,DuanBullo2021}.

A second generalisation addresses model uncertainty: instead of a single transition matrix, one considers a set $\mathcal{T}$ of possible transition matrices, leading to an imprecise Markov chain (IMC)~\cite{DeCooman2009,HermansSkulj2014}. For IMCs, lower and upper hitting times are obtained by minimising or maximising the expected hitting time over all admissible Markov chains. These two generalisations have so far been studied largely independently: weighted travel times for precise chains on the one hand, and (unweighted) expected hitting times for IMCs on the other.

In this paper, we combine these two perspectives by extending the notion of expected hitting times for IMCs to the weighted setting. We introduce and study \emph{weighted imprecise Markov chains} (WIMCs), in which transitions are both uncertain, through a nonempty set $\mathcal{T}$ of transition matrices, and associated with strictly positive weights given by a matrix $W$.
In this context, we introduce lower and upper expected hitting times as the infimum and supremum of the (weighted) hitting time over $\mathcal{T}$, and we characterise these quantities as the unique solutions of nonlinear fixed-point equations.
The key insight is that any weighted hitting time problem can be transformed into an unweighted hitting time problem on an augmented state space. This connection enables us to adapt and reuse existing IMC theory and algorithms to WIMCs: we are able to show that there always exist transition matrices attaining the minimum and maximum in the definition of lower and upper hitting time for WIMCs, that these quantities are characterised as the unique solutions of two nonlinear fixed-point equations, and that we can compute them efficiently adapting Krak's algorithm~\cite{krak2021comphit} for IMCs.

    The remainder of the paper is organised as follows. Section~\ref{sec:preli} recalls basic notation and classical results on hitting times and Markov chains. Section~\ref{sec:weightedMCs} introduces weighted Markov chains and presents key results connecting the unweighted and weighted framework. Section~\ref{sec:wimcs} defines weighted imprecise Markov chains, presents the main fixed-point results, and discusses algorithmic aspects. 
    \section{Preliminaries}\label{sec:preli}
    Let $\X$ be a discrete finite space with cardinality $N\in \N$. A discrete‐time stochastic process on $\X$ is a sequence of random variables \((X_n)_{n\in \N_0}\) taking values in $\X$.
    The process $(X_n)_{n\in \N_0}$ is said to be a \textit{Markov chain} if it satisfies 
\begin{equation}\label{eq:markovprop}
    \PP_X(X_{n+1}=x_{n+1}\mid X_{0:n}=x_{0:n})=\PP_X(X_{n+1}=x_{n+1}\mid X_{n}=x_{n})
\end{equation}
for all $x_0, \dots, x_{n+1} \in \X$ and all $n\in\N_0$, and where we write $\PP_X$ for the probability measure associated to the process $(X_n)_{n\in \N_0}$. The Markov chain $(X_n)_{n\in \N_0}$ is said to be (time-)\textit{homogeneous} if
\begin{equation}
   \PP_X(X_{n+1}=y\mid X_n=x)=\PP_X(X_1=y\mid X_0=x) 
\end{equation}
for all $x,y\in \X$ and all $n\in \N_0$. The $N$ by $N$ stochastic matrix $T$ defined by \[T(x,y)\coloneqq \PP_X(X_1=y\mid X_0=x)\] is the \textit{transition matrix} of the homogeneous Markov chain $(X_n)_{n\in \N_0}$ and uniquely characterises its behaviour up to its initial distribution.
One may view a homogeneous Markov chain as a random walk on the directed graph $G_T=(V,E)$ where $V\coloneqq \X$ and $(x,y)\in E$ if and only if $T(x,y)>0$.

Let $A\subset \X$ be a nonempty target set of states. The \textit{hitting time} of the Markov chain $(X_n)_{n\in \N_0}$ is the random variable defined as
\begin{equation}
      \tau_A :=\inf\{n\ge0: X_n\in A\}\in\N_0\cup\{+\infty\},
\end{equation} 
and, for a starting state $x\in\X$, the \emph{expected hitting time} is defined as
\begin{equation}
  h^T(x):=\EE_{\PP_X}\bigl[\tau_A \mid X_0=x\bigr].
\end{equation}
The quantity $h^{T}(x)$ can be seen as the mean number of steps the chain $(X_n)_{n\in \N_0}$ takes before reaching the target $A$, and is the minimal nonnegative solution of a linear system~\cite[{Theorem 1.3.5}]{norrisbook1997markov}
\begin{equation}\label{eq:system_hittingtimes}
\begin{cases}
    h^T(x)=0&\text{if $x\in A$,}\\
     h^T(x)=1+\sum\limits_{y\in \X}T(x,y)h^T(y)&\text{if $x\notin A$.}
\end{cases}
\end{equation} 
\section{Weighted Markov Chains}\label{sec:weightedMCs}
Let $T$ be a transition matrix on $\X$ and let $W>0$ be a $N$ by $N$ matrix of weights. The quantity $W(x,y)$ can be understood as the time or cost that the homogeneous Markov chain $(X_n)_{n\in \N_0}$ governed by $T$ takes to transition from $x$ to $y$. We call the pair $((X_n)_{n\in \N_0}, W)$ a \textit{weighted Markov chain}. 
We can naturally define the \textit{hitting time} for a weighted Markov chain as~\cite{Carron2016Weighted,DuanBullo2021}
\begin{equation}
    \eta_A^W:=\sum_{n=1}^{\tau_A}W(X_{n-1},X_{n}).
\end{equation}
Conditioned on the chain starting in $x\in\X$, the \emph{expected hitting time} for a weighted Markov chain is
\begin{equation}
  h^{T,W}(x):=\EE_{\PP_X}\bigl[\eta_A^W \mid X_0=x\bigr].
\end{equation}
This quantity represents the expected travel time from a state $x\in \X$ to $A$. Similarly as before, the expected hitting time satisfies a system of equations~\cite{Carron2016Weighted}:
\begin{equation}\label{eq:recweighthitting}
\begin{cases}
    h^{T,W}(x)=0&\text{if $x\in A$,}\\
     h^{T,W}(x)=\sum\limits_{y\in \X} T(x,y)\left(W(x,y) + h^{T,W}(y)\right)&\text{if $x\notin A$.}
\end{cases}
\end{equation}
The next result follows directly from the linearity in $W$ of $h^{T,W}$, as in Eq. \eqref{eq:recweighthitting}. 
\begin{lemma}\label{lemma:rescale}
    Let $c>0$. Then $h^{T,cW}=ch^{T,W}$.
\end{lemma}
The next result builds a bridge between weighted and unweighted Markov chains, i.e, under the condition that $W\ge 2$, for every weighted Markov chain we can construct a (unweighted) Markov chain with the same expected hitting time.
\begin{proposition}\label{prop:unweighting}
    Let $T$ be a transition matrix on $\X$, let $A\subset \X$ be the target, and let $W\ge 2$. We can construct a space $\X'\coloneqq \X\cup \ZZ$ with 
    \[\ZZ\coloneqq \{(x,y)\in \X^2: T(x,y)>0\}\]
     and a transition matrix $T'$ on $\X'$ such that $h^{T,W}(x)=h^{T'}(x)$ for all $x\in \X$, where $h^{T,W}$ and $h^{T'}$ are the vectors of expected hitting times for the weighted (resp. unweighted) Markov chain on $\X$ (resp. $\X'$) governed by $T$ (resp. $T'$).
\end{proposition}
\begin{proof}
    Let $T'$ be the transition matrix on $\X'$ defined as 
    \begin{equation}\label{banana}
        \begin{split}
            & 1) \ \ T'(x,z_{xy})=T(x,y),\\
            & 2) \ \ T'(z_{xy},y)={1}/({W(x,y)-1}),\\
            & 3) \ \ T'(z_{xy},z_{xy})=1-{1}/({W(x,y)-1}),
        \end{split}
    \end{equation}
    for all $x,y\in \X$ and $z_{xy}\in \ZZ$. Since $W\ge 2$, $T'$ is a well-defined transition matrix. The following graph represents the situation locally.
     
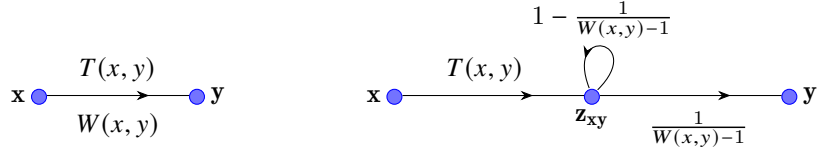
\begin{center}
    \begin{tikzpicture}[
        xscale=1.75,yscale=1.75,
        every node/.style={draw=black,circle,fill=black!70,inner sep=2pt},
        redNode/.style={draw=red, fill=red!50},
        blueNode/.style={draw=blue, fill=blue!55},
        blackNode/.style={draw=black, fill=black},
        every label/.style={rectangle,fill=none,draw=none},
        every edge/.style={draw,bend right,looseness=0.3,
            postaction={
                decorate,
                decoration={
                    markings,
                    mark=at position 0.72 with {\arrow[#1]{Stealth}}
                }
            }
        }
    ]
        \node[blueNode,label=180:$\mathbf{x}$] (x) at (1,50) {};
        \node[blueNode,label=0:$\mathbf{y}$] (y) at (2.2,50) {};
       \node[blueNode,label=180:$\mathbf{x}$] (xp) at (3.7,50) {};
       \node[blueNode,label=270:$\mathbf{z_{xy}}$] (z) at (5.2,50) {};
       \node[blueNode,label=0:$\mathbf{y}$] (yp) at (6.7,50) {};
        \path (x) edge[bend left=0] node[above, draw=none, fill=none, yshift=-8pt] {$T(x,y)$} node[below, draw=none, fill=none, yshift=7.9pt]{$W(x,y)$} (y);
        \path (xp) edge[bend right=0] node[above,draw=none, fill=none, yshift=-8pt, xshift=-3pt] {$T(x,y)$} (z);
        \path (z) edge[relative=false,out=45,in=111,min distance=3.5ex] node[above,draw=none, fill=none, yshift=-20.1pt] {$1-\frac{1}{W(x,y)-1}$} (z);
        \path (z) edge[bend right=0] node[below,draw=none, fill=none, yshift=10pt, xshift=3pt] {$\frac{1}{W(x,y)-1}$} (yp);
    \end{tikzpicture}
    \end{center}
    For all $x\in \X$, the expected hitting time of the Markov chain on $\X'$ governed by $T'$ satisfies
    \begin{align}
    h^{T'}(x)&=1+\sum_{z'\in \X'}T'(x,z')h^{T'}(z')
     = 1+ \sum_{y\in \X} T'(x,z_{xy})h^{T'}(z_{xy}).\label{congobelga}
    \end{align} 
    For all $z_{xy}\in \ZZ$, the expected hitting time satisfies
    \begin{align}
        &h^{T'}(z_{xy}) = 1 + \left(1-\frac{1}{W(x,y)-1}\right)h^{T'}(z_{xy})+\left(\frac{1}{W(x,y)-1}\right)h^{T'}(y)\nonumber\\
        &\quad\Rightarrow \quad h^{T'}(z_{xy})=h^{T'}(y)+ W(x,y)-1.\label{sanmarino}
    \end{align}
    We substitute \eqref{sanmarino} into \eqref{congobelga} and obtain
    \begin{align*}
         h^{T'}(x)&=1+ \sum_{y\in \X} T'(x,z_{xy})h^{T'}(z_{xy})\\
         & = 1+\sum_{y\in \X} T(x,y)\left(h^{T'}(y)+ W(x,y)-1\right)\\
         & = \sum_{y\in \X} T(x,y)\left(W(x,y)+ h^{T'}(y)\right),
    \end{align*}
     which is exactly the system in Eq. \eqref{eq:recweighthitting}. Then, $h^{T'}(x)=h^{T,W}(x)$ for all $x\in \X$.\qed
\end{proof}
The following result extends Proposition \ref{prop:unweighting} to the general case $W>0$.
\begin{corollary}\label{cor:superunweights}
    Let $T$ be a transition matrix on $\X$, let $A\subset \X$ be the target, and let $W(x,y)>0$ for all $x,y\in \X$. We can construct a space $\X'\coloneqq X\cup \ZZ$ 
    and a transition matrix $T'$ on $\X'$ as in Proposition \ref{prop:unweighting} so that, for some $c\ge1$, we have $h^{T,W}(x)=\frac{1}{c}h^{T'}(x)$ for all $x\in \X$, 
    where $h^{T,W}$ and $h^{T'}$ are the vectors of expected hitting times for the weighted (resp. unweighted) Markov chain on $\X$ (resp. $\X'$) governed by $T$ (resp. $T'$).  
\end{corollary}
\begin{proof}
    Since $W>0$, we have $w_m\coloneqq \min_{x,y\in \X} W(x,y)>0$. Define $\hat{W}\coloneqq cW$, where $c\coloneqq \max\{1,\nicefrac{2}{w_m}\}$. It follows that $\hat{W}\ge 2$ and, by Lemma \ref{lemma:rescale}, we have $ch^{T,W}=h^{T,cW}=h^{T,\hat{W}}$. 
    By following the construction in the proof of Proposition \ref{prop:unweighting}, we can find a transition matrix $T'$ on $\X'$ such that $h^{T'}(x)=h^{T,\hat{W}}(x)=ch^{T,W}(x)$ for all $x\in \X$, which is what we wanted to prove. \qed
\end{proof}

\section{Weighted Imprecise Markov Chains}\label{sec:wimcs}
Instead of a single homogeneous Markov chain $(X_n)_{n\in \N_0}$ with a fixed transition matrix $T$, we now consider a set $\T$ of admissible transition matrices. Using this set, we define the family $\Pcal$ that contains all homogeneous Markov chains\footnote{By Krak et al.~\cite{krak2019hitting}, lower and upper hitting times remain unchanged if $\Pcal$ is enlarged to include all (time-)inhomogeneous non-Markov processes compatible with $\T$.} whose transition matrix is in $\T$. This set $\Pcal$ is called an \textit{imprecise Markov chain} (IMC)~\cite{DeCooman2009,HermansSkulj2014,hartfiel_seneta_1994}.  
Throughout this paper, we impose some conditions on the set $\T$: we assume that it is nonempty, compact, convex, and has separately specified rows (SSR)~\cite{HermansSkulj2014,krak2019hitting,sangalli2026computing}.  This last property means that the set $\T$ is the Cartesian product of $N=|\X|$ compact and convex sets of probability distributions $\{\T_x\}_{x\in\X}$, one for each state $x\in\X$.

Given a set of transition matrices $\T$ on $\X$, we define the \textit{lower and upper expected hitting times} as follows:
\begin{equation*}
    \underline{h}^\T \coloneqq \inf_{T\in \T} h^T \quad\text{and}\quad  \overline{h}^\T \coloneqq \sup_{T\in \T} h^T.
\end{equation*}
Krak~\cite{krak2021comphit} showed that, under the 
condition
\begin{description}
    \item[(R1):] for all $x\in \X$ there exists $n\in \N_0$ such that $[\underline{T}^n\one_A](x)>0$,  
\end{description}
i.e that the process can lower reach $A$ regardless of the starting state,
the vector of lower hitting times is finite and the unique solution of a system of equations:
\begin{equation}
    \underline{h}^\T=\one_{A^c}+ \one_{A^c}\cdot \underline{T} \ \underline{h}^\T,
\end{equation}
where $\one$ is the indicator function and $\underline{T}:\R^\X\to \R^\X$ is the lower transition operator associated to $\T$, given by:
\begin{equation}
    [\underline{T}f](x)\coloneqq \inf_{T\in \T} [Tf](x), \text{ for all $f\in \R^\X$.}
\end{equation}
A completely analogous characterisation also holds for upper hitting times.

Let $W>0$ be a matrix of weights. We define the pair made of an IMC $\Pcal$ and a weight matrix $W$ as a \textit{weighted imprecise Markov chain} (WIMC). Given a set of transition matrices $\T$, we define the lower and upper expected hitting times for a WIMC as
\begin{equation}
    \underline{h}^W:= \inf_{T\in \T} h^{T,W}\quad \text{and} \quad \overline{h}^W:= \sup_{T\in \T} h^{T,W}.
\end{equation}
We want to characterise these quantities as the unique solutions of two nonlinear fixed-point equations.
We begin with the following lemma.
\begin{lemma}
    Let $\T$ be a nonempty, compact, convex, with SSR set of transition matrices on $\X$. Then, there exist $T,\Tilde{T}\in \T$ such that 
    \begin{equation}
        h^{T,W}=\underline{h}^W\quad \text{and}\quad h^{\Tilde{T},W}=\overline{h}^W.
    \end{equation}
\end{lemma}
\begin{proof}
    Using the notation of Proposition \ref{prop:unweighting} and Corollary \ref{cor:superunweights}, we construct a set of transition matrices $\T'$ on $\X'=\X\cup \ZZ$ containing all matrices $T'$ defined starting from a matrix $T\in \T$. The set $\T'$ is nonempty, compact, convex, and has SSR. 
    By Krak et al.~\cite[Theorem 12]{krak2019hitting}, there exists a matrix $T' \in\T'$ that achieves the minimum in the definition of lower hitting time. 
    Therefore, the corresponding matrix $T\in \T$ on $\X$ defined by $T(x,y)=T'(x,z_{xy})$ achieves the minimum in $\underline{h}^W$. An analogous argument holds for upper hitting times. \qed
\end{proof}

Define the operator $\underline{\Lscr}^W:\R^\X\to \R^\X$ as
\begin{equation}\label{defnlscr}
(\underline{\Lscr}^W f)(x)=
\begin{cases}
    0&\text{if $x\in A$,}\\
    \inf\limits_{T\in \T} \Bigl[ \sum\limits_{y\in \X} T(x,y)\bigl(W(x,y) + f(y) \bigr) \Bigr]&\text{if $x\notin A$,}
\end{cases}
\end{equation}
for all $f\in \R^\X$.
The following theorem states that $\underline{h}^W$ is the unique solution of a nonlinear fixed-point equation.
\begin{theorem}
Let $\T$ be a nonempty, compact, convex, with SSR set of transition matrices on $\X$. 
Then, under condition (R1), $\underline{h}^W$ is the unique solution of
\begin{equation}
    \underline{h}^W=\underline{\Lscr}^W \underline{h}^W.
\end{equation}
\end{theorem}
\begin{proof}
    If $W\ge2$, we follow the construction and notation of Proposition \ref{prop:unweighting} and we obtain an (unweighted) imprecise Markov chain on $\X'=\X\cup\ZZ$ with set of transition matrices $\T'$.
    We know from Krak~\cite[Proposition 2]{krak2021comphit} that there exists a unique $\underline{h}^{\T'}\in \R^{\X'}$ that satisfies 
    \begin{equation*}
    \underline{h}^{\T'}(x)=
        \begin{cases}
            0&\quad\text{if $x\in A$,}\\
            1+\inf\limits_{T'\in \T'} [T'\underline{h}^{\T'}](x) &\quad\text{if $x\notin A$.}
        \end{cases}
    \end{equation*}
    In particular, there exists a matrix $\hat{T}'\in \T'$ satisfying $h^{\hat{T}'}=\underline{h}^{\T'}$. The transition matrix $\hat{T}\in \T$ defined as $\hat{T}(x,y)=\hat{T}'(x,z_{xy})$ for all $x,y\in \X$ achieves the minimum:
    \begin{equation*}
        h^{\hat{T},W}(x)=h^{\hat{T}'}(x)= \underline{h}^{\T'}(x)=\inf_{T'\in \T'}h^{T'}(x)= \inf_{T\in \T} h^{T,W}(x)=\underline{h}^W(x),
    \end{equation*}
    for all $x\in \X$.
    It follows that
    \begin{align*}
        \underline{h}^W(x)&=h^{\hat{T},W}(x)=h^{\hat{T}'}(x)=\underline{h}^{\T'}(x)= 1+\inf_{T'\in \T'} [T'\underline{h}^{\T'}](x)\\
        &=1+\inf_{T'\in \T'} \sum_{y\in \X}T'(x,z_{xy}) \underline{h}^{\T'}(z_{xy})\\
        &=1+\inf_{T'\in \T'} \sum_{y\in \X}T'(x,z_{xy})\left(\underline{h}^{\T'}(y)+W(x,y)-1\right)\\
        &=\inf_{T\in \T} \sum_{y\in \X}T(x,y)\left(\underline{h}^W(y)+W(x,y)\right)=\left[\underline{\Lscr}^W\underline{h}^W\right](x),
    \end{align*}
    for all $x\notin A$. Trivially, the equality $\underline{h}^W(x)=[\underline{\Lscr}^W\underline{h}^W](x)$ also holds for all $x\in A$. We note that $\underline{h}^{\T'}(z_{xy})=\underline{h}^{\T'}(y)+W(x,y)-1$ follows directly from Eq. \eqref{sanmarino} in the proof of Proposition \ref{prop:unweighting}. 

On the other hand, if $0<W(x,y)<2$ for some $x,y\in \X$ we rescale all the weights with a constant $c\ge 1$ as in Corollary \ref{cor:superunweights}. We obtain
\begin{align*}
&c\underline{h}^W(x)=\underline{h}^{cW}(x)=\left[\underline{\Lscr}^{cW}\underline{h}^{cW}\right](x)= \left[\underline{\Lscr}^{cW}\left(c\underline{h}^{W}\right)\right](x) \\
&\qquad= \inf_{T\in \T} \sum_{y\in \T} T(x,y)\left( cW(x,y)+  c \underline{h}^{W}(y)\right)\\
&\qquad= c\inf_{T\in \T} \sum_{y\in \T} T(x,y)\left( W(x,y)+  \underline{h}^{W}(y)\right) =c\left[\underline{\Lscr}^W\underline{h}^W\right](x),
\end{align*}
for all $x\notin A$, while equality is trivial for all $x\in A$. \qed
\end{proof}
By defining $\overline{\Lscr}^W$ taking the supremum in place of the infimum in \eqref{defnlscr}, one can show that the upper hitting time $\overline{h}^W$ is the unique solution of a fixed-point equation:
\begin{equation}
    \overline{h}^W=\overline{\Lscr}^W \overline{h}^W.
\end{equation}

\subsection{Computing Lower and Upper Hitting Times}
Krak~\cite{krak2021comphit} proposes efficient iterative algorithms to compute lower and upper hitting times for (unweighted) imprecise Markov chains. Thanks to Proposition \ref{prop:unweighting} and Corollary \ref{cor:superunweights}, we can use the same algorithms to compute lower and upper hitting times for weighted imprecise Markov chains. In particular, the algorithm that computes lower hitting times for WIMCs alternates between computing the (weighted) hitting time $h^{T_n}$ by solving the linear system in Eq. \eqref{eq:recweighthitting} for a matrix $T_n\in \T$, and improving the transition matrix by solving
\begin{equation}\label{coniglio}
T_{n+1}(x,\cdot)=\argmin_{T\in \T} \sum\limits_{y\in \X} T(x,y)\left(W(x,y)+ h^{T_n}(y)\right),
\end{equation}
for all $x\in A^c$. 
This procedure is analogous to the one presented by Krak~\cite{krak2021comphit} for (unweighted) IMCs. In particular, since the matrices in $\T$ and $\T'$ are in one-to-one correspondence, we may transform the WIMC into an unweighted IMC as in Corollary \ref{cor:superunweights} and apply Krak's algorithm to $\T'$. This algorithm would carry out the same computations as the one we propose, and each intermediate matrix $T_{n+1}'$ would correspond to $T_{n+1}$ in Eq. \eqref{coniglio}.
Therefore, the convergence of this algorithm (and the analogous variant for upper hitting times) is guaranteed by Krak~\cite{krak2021comphit}.
The computational complexity of this algorithm is the same as that of Krak~\cite{krak2021comphit}: each iteration requires at least $O(N^3)$ time, while the number of iterations depends on the geometry of $\T$. For a detailed complexity analysis, we refer the reader to Krak~\cite{krak2021comphit}.
\subsubsection*{Acknowledgements}
This work has been partly supported by the PersOn project
(P21-03), which has received funding from Nederlandse Organisatie voor Wetenschappelijk Onderzoek (NWO). The authors thank the two anonymous reviewers for their careful reading and valuable suggestions that helped improve the exposition.

\bibliographystyle{spbasic.bst}

\end{document}